\documentclass[12pt]{amsproc}
\usepackage{amssymb}
\usepackage{listings}

\newtheorem{theorem}{\bf Theorem}[section]
\newtheorem{lemma}[theorem]{\bf Lemma}
\newtheorem{proposition}[theorem]{\bf Proposition}
\newtheorem{corollary}[theorem]{\bf Corollary}
\newtheorem{remark}[theorem]{\sc Remark}

\DeclareMathOperator{\GF}{GF}
\DeclareMathOperator{\AGL}{AGL}

\DeclareMathOperator{\diam}{diam}

\DeclareMathOperator{\PSL}{PSL}
\DeclareMathOperator{\sz}{Sz}
\DeclareMathOperator{\gl}{GL}
 \DeclareMathOperator{\soc}{soc}
  \DeclareMathOperator{\Aut}{Aut}

 \newcommand{\gen}[1]{\left\langle#1\right\rangle} 
 \newcommand{\nn}{\mathrel{\unlhd}}

\DeclareMathOperator{\B}{B}

\begin{document}
\title{The Engel graph of a finite group }

\author{Eloisa Detomi}
\address{E. Detomi, Dipartimento di Ingegneria dell'Informazione - DEI, Universit\`a 
di Padova, Via G. Gradenigo 6/B, 35121 Padova, Italy} 
\email{eloisa.detomi@unipd.it}

\author{Andrea Lucchini}
\address{A. Lucchini, Universit\`a di Padova, Dipartimento di Matematica ``Tullio Levi-Civita", Via Trieste 63, 35121 Padova, Italy}
\email{lucchini@math.unipd.it}

\author{Daniele Nemmi}
\address{D. Nemmi, Universit\`a di Padova, Dipartimento di Matematica ``Tullio Levi-Civita", Via Trieste 63, 35121 Padova, Italy}
\email{dnemmi@math.unipd.it}

\begin{abstract}
For a finite group $G,$ we investigate the direct graph $\Gamma(G),$ whose vertices are the non-hypercentral elements of $G$ and where there is an
edge $x\mapsto y$ if and only if $[x,_ny]=1$ for some $n \in \mathbb N.$
We prove that $\Gamma(G)$ is always weakly connected and is strongly connected if $G/Z_{\infty}(G)$ is neither Frobenius nor almost simple. 
\end{abstract}

\maketitle

\section{Introduction}

Let $G$ be a finite group. The commuting graph of $G$ has
as vertices  the non-central elements of $G$ and two distinct vertices
are adjacent if and only if they commute in $G$. Commuting graphs arise naturally in many
different contexts and they have been intensively studied by various authors in recent
years  (see in particular  \cite{ij}, \cite{GP},  \cite{mp}, \cite{parker}, \cite{SS}).

The commuting graph of a group can be viewed  
 through a different lens, which leads naturally to some interesting generalisations.  To do this, first let $\mathcal{A}$ be the class of abelian groups and let $\Lambda_{\mathcal{A}}(G)$ be the graph with vertex set $G$ so that $x$ and $y$ are adjacent if and only if the subgroup $\langle x, y \rangle$ of $G$ is abelian. Clearly, every vertex in the center $Z(G)$ is connected to every other vertex in this graph (i.e. it is a universal vertex of the graph), so it makes sense to consider the more restrictive graph $\Gamma_{\mathcal{A}}(G)$, which is only defined on the  non-central elements of $G$. Then $\Gamma_{\mathcal{A}}(G)$ is the commuting graph of $G$ defined above.  More in general, let $\mathcal{F}$ be  a family of groups and consider the graph $\Lambda_{\mathcal{F}}(G)$ on the elements of $G$, where distinct vertices $x$ and $y$ are adjacent if and only if $\langle x,y \rangle$ is in $\mathcal{F}$. As before, we would like to define the related graph $\Gamma_{\mathcal{F}}(G)$ on $G \setminus I_{\mathcal{F}}(G)$, where $I_{\mathcal{F}}(G)$ is the set of universal vertices of $\Lambda_{\mathcal{F}}(G)$. If $\mathcal{F} = \mathcal{N}$ is the class of nilpotent groups, then $I_{\mathcal{F}}(G)$ coincides with the hypercenter $Z_\infty(G)$ of $G$ 
 which is the final term in the upper central series of $G$ (see \cite[Proposition 2.1]{az}) and $\Gamma_{\mathcal{F}}(G)$ is called 
  the nilpotent graph. Similarly, if $\mathcal{F} = \mathcal{S}$ is the class of soluble  groups, then \cite[Theorem 1.1]{gu} implies that $I_{\mathcal{F}}(G) = R(G)$ is the soluble radical and $\Gamma_{\mathcal{F}}(G)$ is the soluble graph of $G$. Notice that if $Z(G)=1,$ then the nilpotent graph $\Gamma_{\mathcal{N}}(G)$ is connected if and only if 
the abelian graph $\Gamma_{\mathcal{A}}(G)$ is connected and $\diam(\Gamma_{\mathcal{A}}(G))\leq 2\cdot \diam(\Gamma_{\mathcal{N}}(G)).$ Indeed if
$x$ and $y$ are non-trivial elements of $G$  and $\langle x, y \rangle$ is nilpotent, then  $\langle x, y \rangle$ contains a non-trivial central element $z$ and $(x,z,y)$ is a path in   $\Gamma_{\mathcal{A}}(G).$ 

The notion of commuting graph can be also generalized in a different direction. Let $w$ be a word in the free group $F_2$ of rank 2.  We may define a direct graph $\Lambda_w(G)$ on the elements of $G,$ where there is an edge $x\mapsto y$ if and only if $w(x,y)=1.$
Set  
$$\begin{aligned}I_{r,w}(G)&=\{g\in G \mid w(g,x)=1 \text { for all }x \in G\},\\
	I_{l,w}(G)&=\{g\in G \mid w(x,g)=1 \text { for all }x \in G\},\\
	I_w(G)&=I_{r,w}(G) \cap I_{l,w}(G).
\end{aligned}$$
Noticing that $I_{{w}}(G)$ is the set of universal vertices of $\Lambda_{w}(G)$, we define the related graph $\Gamma_{w}(G)$ on $G \setminus I_{w}(G)$.

Set $[x,{}_0 y]=x$ and $[x,{}_{i+1}y]=[[x,{}_iy],y]$ for $i\geq 0$. The word $[x,{}_ny]$ is called the $n$-th Engel word. We will use the symbols $\Lambda_n(G),\Gamma_n(G),$ $I_{r,n}(G), I_{l,n}(G)$ to denote the graphs $\Lambda_w(G)$ and $\Gamma_w(G)$ and the sets  $I_{r,w}(G), I_{l,w}(G),$
when $w=[x,_ny]$ is the $n$-Engel word. Notice
that $\Gamma_1(G)$ coincides with the commuting graph of $G.$
	
	In this paper we want to study  the graph $\Lambda(G)=\cup_{n > 0}\Lambda_n(G),$
		 i.e. the graph whose vertices are the elements of $G$ and where there is an edge $x\mapsto y$ if and only if $[x,_ny]=1$ for some $0< n \in \mathbb N.$

The union $I_l(G)=\cup_n I_{l,n}(G)$ is the set of the left Engel elements and coincides with the Fitting subgroup $F(G)$ of $G.$ The union $I_r(G)=\cup_n I_{r,n}(G)$ is the set of the right Engel elements and coincides with the hypercenter $Z_\infty(G)$ of $G$  (see e.g. \cite[12.3.7]{rob}). 
 In particular the  elements of $G$ that are connected to every other vertex (both as starting and ending vertex of some edge) are the   elements of
$I_l(G)\cap I_r(G)=Z_\infty(G).$  So it makes sense to consider the more restrictive graph $\Gamma(G)$, which is only defined on the  non-hypercentral elements of $G$.
We will call $\Gamma(G)$ the Engel graph of $G.$ 

Our aim is to study the connectivity properties of the Engel graph. 
Since  $\Gamma(G)$  is a direct graph we may consider the strong connectivity and the weak connectivity.
 A directed graph is {\slshape{strongly connected }} if it contains a directed path from $x$ to $y$ (and from $y$ to $x$) for every pair of vertices $(x, y).$ The strong components are the maximal strongly connected subgraphs. A directed graph is {\slshape{weakly connected}}  if the undirected underlying graph obtained by replacing all directed edges of the graph with undirected edges is a connected graph. 
Our first main result is the following.

\begin{theorem}\label{wcon}
	If $G$ is a finite group, then $\Gamma(G)$ is weakly connected and its undirected diameter is at most $10.$
\end{theorem}

The study of the strong connectivity of   $\Gamma(G)$ is more complicated. It can be easily proved that 	$\Gamma(G)$ is strongly connected if and only if $\Gamma(G/Z_\infty(G))$ is  strongly connected (see Lemma \ref{reduce}) and that 
 the Engel graph of a Frobenius group is not strongly connected (Lemma \ref{frob-graph}). 
 Our main result is the following. 
 
\begin{theorem}\label{mainth}
	Suppose that $G/Z_\infty(G)$ is not an almost simple group. Then 
 $\Gamma(G)$ is strongly connected if and only if $G/Z_\infty(G)$ is not a Frobenius group.
\end{theorem}

 The case where  $G/Z_\infty(G)$ is an almost simple group is much more difficult to handle. Indeed, 
   there  are infinitely many simple groups whose Engel graph is strongly connected, for example the alternating groups $A_n$ if $n\geq 6$ (see Subsection \ref{alterni}), but also infinitely many simple groups whose Engel graph is not strongly connected, for example $\PSL(2,q)$ when $q$ is even and the Suzuki groups $\sz(q)$ (see Subsection \ref{sconnessi}). 
 An immediate application of
 Lemma \ref{norma} is that if $G$ is almost simple and $\Gamma(\soc(G))$ is strongly connected, then $\Gamma(G)$
 is also strongly connected. The converse is not always true: for example $\Gamma(S_5)$ is strongly connected, but $\Gamma(A_5)$ is not.

 In the case where $G$ is soluble and $G/Z_\infty(G)$ is not a Frobenius group, we can also bound the diameter of  $\Gamma(G)$.  
 
 \begin{theorem}\label{diam4}
	Suppose that $G$ is soluble and $G/Z_\infty(G)$ is not a Frobenius group. Then 
  $\diam\Gamma(G)\leq 4$. Furthermore, there exists a soluble group $G$ such that $\diam\Gamma(G) = 4$.
\end{theorem}
 
It is interesting to note that if $Z_\infty(G)=1$ and $G$ is not almost simple, then $\Gamma(G)$ is strongly connected if and only if $\Gamma_2(G)$ is strongly connected (see Theorem \ref{gamma2}). This is no more true if the assumption $Z_\infty(G)=1$ is removed. For example $\Gamma_3(\gl(2,3))$
is strongly connected, but $\Gamma_2(\gl(2,3))$ is not (see Remark \ref{osser}).

The nilpotent graph and the Engel graph of $G$ have the same vertex-set, $G \setminus Z_\infty(G). $ As all elements of a nilpotent group are (left and right) Engel, if $g_1$ and $g_2$ are adjacent in the nilpotent graph, then $g_1\mapsto g_2$ and $g_2\mapsto g_1$ in the Engel graph. So if the nilpotent graph of
$G$ is connected, then the Engel graph $\Gamma(G)$ is strongly connected. The converse is false, indeed if $G$ is $S_4$, and more in general a 2-Frobenius group, or $A_p$, when $p$ is a prime with $p>5$, then the commuting graph of $G$ is not connected (see \cite{parker}), and the same holds for its nilpotent graph, as mentioned above. On the other hand, it follows from Theorem \ref{mainth} and Subsection \ref{alterni} that the Engel graph of $G$ is  strongly connected. 

The paper is organized as follows. In Section \ref{sec:2} we collect some preliminary results about the Engel graph and prove Theorem \ref{wcon}. In Section \ref{sec:3} we prove Theorem \ref{mainth}. The proof of Theorem \ref{diam4} will be completed in Section \ref{sec:4}. In Section \ref{sec:5} we analyze the  Engel graph of some families of simple and almost simple groups. Finally, in Section \ref{sec:6}, with the help  of some computer calculations,  we prove that the Engel graph of a sporadic simple group is strongly connected. 

\subsection*{Acknowledgements} We would like to thank  Pablo Spiga for fruitful discussions and  his precious help in the computations with some sporadic simple groups.

\section{Preliminary results and proof of Theorem \ref{wcon}}\label{sec:2}

In this and in the following sections we will use the symbols
$x\mapsto  y$ and $x\mapsto_n y$ to say there there is an edge from $x$ to $y$ in the graph $\Lambda(G)$, or respectively $\Lambda_n(G)$,  and the
 symbols $x \twoheadrightarrow    y$ 
and $x \twoheadrightarrow_n    y$
to say that there is a direct path in the graph $\Gamma(G)$, or respectively $\Gamma_n(G)$, joining $x$ to $y.$ 
 Recall that $\Gamma(G)$ and $\Gamma_n(G)$  are the subgraphs  
  defined on the set of non-universal vertices of $\Lambda(G)$ and $\Lambda_n(G)$, respectively. 

\begin{lemma}\label{reduce}
	$\Gamma(G)$ is weakly or strongly connected if and only if $\Gamma(G/Z_\infty(G))$ is weakly or strongly connected.
\end{lemma}	
\begin{proof}
If $y$ is adjacent to $x$ in $\Gamma(G)$, then its image $\bar y$ in $G/Z_\infty(G)$ is adjacent to $\bar x$ in $\Gamma(G/Z_\infty(G))$. Conversely, if 
$\bar y \in G/Z_\infty(G)$ is adjacent to $\bar x$ in $\Gamma(G/Z_\infty(G))$, then $[y, _n x] \in Z_\infty(G)= Z_m(G)$ for some $n$ and $m$. Thus 
$[ y, _{n+m} x]=[[y, _n x],_m x] =1$ and $y$ is adjacent to $x$ in $\Gamma(G)$. 
\end{proof}

\begin{lemma}\label{trivial}  Let $x, y \in G.$ 
\begin{enumerate}
\item If $[y,_n x] = [y,_m x] \neq 1$ for some  $0 \le n <m$, then $y$ is not adjacent  to $x$ in $\Gamma(G).$
\item If $y \in F(G)$, then $x \mapsto y$ for every $x \in G$.
\item If $y \in N_G(\langle x \rangle)$,  then $y \mapsto_2 x$.
\end{enumerate}
\end{lemma}
\begin{proof}
\begin{enumerate}
\item Assume $1 \neq [y,_n x] = [y,_m x]$ for some  $0\le n <m$ and let $z=[y,_n x]$. Then $z=[z,_{m-n} x]$ and so $z=[z,_{l(m-n)} x]$ for every $l \ge 0$. If $[y,_N x]=1$, then $1= [y,_{N} x, _{N(m-n-1)} x ]=[y,_{N(m-n)} x] =z,$ a contradiction.
\item It follows from the fact that the Fitting subgroup $F=F(G)$ coincides with the set of the left Engel elements of $G$.
\item If $y \in N_G(\langle x \rangle)$, then $[y,x] \in \langle x \rangle$ and it commutes with $x$, so $[y,_2 x]=1$.\qedhere
\end{enumerate}
\end{proof}

The proof of Theorem \ref{wcon}  relies on  
the main result of \cite{bgl}, which states that the soluble graph of an insoluble group is connected and its diameter is at most $5$. 

\begin{proof}[Proof of Theorem \ref{wcon}] By Lemma \ref{reduce}, we may assume that $Z_\infty(G)=1.$ If the soluble radical $R(G)$ is non-trivial, then $F(G)\neq 1$ and $\Gamma(G)$ is weakly connected by Lemma \ref{trivial} (2). So we may assume $R(G)=1.$ Let $x_1, x_2$ be two distinct non-identical elements of $G$. If $\langle x_1, x_2 \rangle$ is soluble then
	there exists $1\neq f \in F(\langle x_1, x_2 \rangle)$ and $x_1 \mapsto f, x_2 \mapsto f.$ 
	Since the soluble graph of $G$ is connected with diameter at most $5$ (see \cite{bgl}), we deduce that 
	$\Gamma(G)$ is weakly connected and its undirected diameter is at most $10$, as claimed.
\end{proof}

Note that if $\Omega$ is a strong component of $\Gamma(G)$ and a non-hypercentral element $x\in G$ commutes with some element of $\Omega$, then $x \in \Omega$. Thus the following result will be relevant for our proofs. 
\begin{theorem}\label{Thompson}\cite[Theorem 10.2.1]{gor} If a finite group $G$ admits a fixed-point-free automorphism of prime order,
then $G$ is nilpotent.
\end{theorem} 

We will also need the following result concerning Frobenius complements.

\begin{lemma}\cite[Lemma 2.1]{mp}\label{frob}
Suppose $X$ is a Frobenius complement. Then every Sylow subgroup of $X$ is cyclic or generalized quaternion. If $X$ has odd order then any two elements of prime order commute and $X$ is metacyclic. If $X$ has even order, then $X$ contains a unique involution.
\end{lemma}

\section{Proof of Theorem \ref{mainth}}\label{sec:3}

\begin{lemma}\label{frob-graph}
	If $G=K\rtimes H$  is a Frobenius group with kernel $K$ and complement $H$, then $\Gamma(G)$ is not strongly connected. In particular for every nontrivial $k \in K$ there is no element $g \in  G \setminus K$ with $k \mapsto g$.
\end{lemma}

\begin{proof}
	Let $k$ be a nontrivial element of $K$ and assume that there exists $g \in G\setminus K$ with $k \mapsto g$. Thus $[k, _n g]=1$ for some integer $n.$ 
	Take $g\in G\setminus K$ such that $n$ is minimal. 
	Since $g \in G\setminus K = \cup_{k \in K} H^k$, without loss of generality we can assume that $g$ belongs to  $H$. So $z=[k, _{n-1} g]$ belongs to $K$, since $K$ is normal, and it belongs to $H$, since it centralizes $g \in H$ and $G$ is Frobenius. Therefore $z=1$, a contradiction to the minimality of $n$. 
\end{proof}

For the remaining part of this section our aim is to prove that if $G$ is neither Frobenius nor almost simple, then $\Gamma(G)$ is strongly connected.
In the following we will denote by $\Delta(G)$ the subgraph of $\Lambda(G)$ induced by the non-trivial elements of $G.$ Clearly $\Gamma(G)=\Delta(G)$ if $Z_\infty(G)=1.$

\begin{lemma}\label{norma} Let $N$ be a non-nilpotent normal subgroup of a finite group $G$. If $\Delta(N)$ is strongly connected, 
 or, more generally, if $N\setminus\{1\}$ is contained in a strong component of $\Delta(G)$, 
then $\Delta(G)$ is strongly connected.
\end{lemma}

\begin{proof}
Let  $\Omega$ be the strong component of $\Delta(G)$ containing $N \setminus \{1\}.$
Let $x$ be a non-trivial element of $G$ and let $\tilde x$ be an element of prime order in $\langle x \rangle$. Since $N$ is not nilpotent,
 it follows from Theorem \ref{Thompson} that 
 there exists $y \in N \setminus \{1\}$  such that $[y, \tilde x]=1.$ Clearly $x$, $\tilde x$ and $y $ form a clique in the commuting graph of $G$, and consequently  $x \in \Omega.$
\end{proof}

\begin{lemma}\label{prodo}Suppose that $N=G_1G_2$ is the central product of two non-trivial finite groups. Then $\Delta(N)$ is strongly connected and
	$\diam(N)\leq 3.$
\end{lemma}
\begin{proof}Let $g_1=x_1y_1, g_2=x_2y_2$ be two non trivial elements of $N$ with	$x_1, x_2 \in G_1, y_1, y_2 \in G_2.$	We may assume $x_1\neq 1.$	If $y_2\neq 1$, then $g_1=x_1y_1, x_1, y_2, x_2y_2=g_2$ is a path in $\Delta(N).$	Suppose $y_2=1$. This implies $x_2\neq 1.$ If $y_1\neq 1,$ then $g_1=x_1y_1, y_1, x_2, x_2y_2=g_2$ is a path in $\Delta(N)^{}.$ We remain with the case $y_1=y_2=1.$ Then, let $1\neq z \in G_2.$ In this case
 $g_1=x_1, z, x_2=g_2$
  is a path in $\Delta(N).$ \end{proof}

In the following we will denote by  $F^*(G)$ the generalized Fitting subgroup of $G$.

\begin{lemma}\label{F*}
 If  $G$ is not an almost simple group and $F^*(G) > F(G)$, then $\Delta(G)$ is strongly connected.
\end{lemma}
\begin{proof} Recall that $F^*(G)=E(G)F(G)$ is the central product of the Fitting subgroup of $G$ with the layer $E(G)$  of $G$. Since $E(G)$ is the central product of the components, either $F^*(G)$ is a non-abelian simple group or $\Delta(F^*(G))$ is strongly connected by Lemma \ref{prodo}. In the first case $G$ is almost simple, in the second $\Delta(G)$ is strongly connected by Lemma \ref{norma}.
\end{proof}

In the following, given two vertices $x, y$ of the graph $\Gamma(G),$  we will denote by $d(x,y)$ the direct distance from $x$ to $y$ in  $\Gamma(G).$

\begin{lemma}\label{normalizer}Let $x, y$ be two distinct non-trivial elements of a finite group $G$. If $Z_\infty(G)=1$, $F(G)=F^*(G)$ and $G$ is not a Frobenius group, then the following hold:
\begin{enumerate}
\item If
$d(x,y) >3$, then $N_G( \langle y \rangle)F$ (and consequently $C_G( \langle y \rangle)F$) is a Frobenius group. 
\item  If $d(x,y) >4$ and $y_r$ is a power of $y$ of order a prime $r$, 
 then $C_G(y_r)$ has odd order and it is metacyclic. 
\end{enumerate}
\end{lemma}
\begin{proof} Assume $d(x,y)>3$. If $F(G) \cap N_G( \langle y \rangle) \neq 1$, then there exists a non trivial element $g \in F(G) \cap N_G( \langle y \rangle) $, so by Lemma \ref{trivial} (2) and (3),
\[ x  \mapsto g  \mapsto y \] 
a contradiction. 
So $F(G) \cap N_G( \langle y \rangle) = 1$.  
Assume, by contradiction, that 
 $N_G( \langle y \rangle)F(G)$ is not a Frobenius group. Then 
 there exists an element $1 \neq g \in C_{F(G)}(k)$ for some  $1 \neq k \in N_G( \langle y\rangle)$ and 
 \[ x \mapsto g \mapsto k \mapsto y\]
against $d(x,y)>3.$

 Now assume  $d(x,y)>4$. In particular, $d(x,y_r)>3$. 
 By (1), $C_G(y_r)$ is a Frobenius complement. Assume that $|C_G(y_r)|$ is even and take an involution $z \in C_G(y_r)$. As every non-trivial element of $C_G(y_r)$ acts fixed-point-freely on $F,$ $z$ is the unique involution of 
 $C_G(y_r)$. But then,  for every $g \in C_G(y_r)$, we have $z^g \in C_G(y_r)$, hence $z^g=z$. Since $z$ acts fixed-point-freely on $F(G)$, it inverts the elements of $F(G)$ and $F(G)$ is abelian. In particular $F(G)=[F(G),z]$ as a set. We claim:
$$(*)\quad G= C_G(z)F(G).$$ 
Indeed let $g \in G$. As $z^g$ acts as involution on $F(G)$, then $[z,g] $ centralizes $F$. Since $C_G(F^*(G))\leq F^*(G)=F(G),$ it follows $[z,g] \in F(G) =[z,F(G)]$. Thus there exists an element $f \in F(G)$ such that $[z,g]=[z,f]$. So $z^g=z^f$, hence $g \in C_G(z)F(G)$. 
 It follows that $G=C_G(z)F(G)$ is not a Frobenius group. Since $C_G(z) \cap F(G)=1,$  
there exists a non-trivial element $g \in C_G(z)$ such that $C_F(g) \neq 1$. Thus, for some $1 \neq f \in C_F(g)$, since $y \in C_G(y_r)$ centralizes $z$, we have 
 \[ x \mapsto f \mapsto g \mapsto z  \mapsto y\]
 a contradiction. Thus $C_G(y_r)$ is a Frobenius complement with odd order  and our conclusion follows from Lemma \ref{frob}.
\end{proof}

\begin{corollary}\label{invo}
	Assume that $Z_\infty(G)=1$ and $G$ is neither a Frobenius group nor an almost simple group. Then all the elements whose centralizer has even order belong to the same strong component of $\Gamma(G).$
\end{corollary}

\begin{proof}
	It suffices to prove that if $|x|=|y|=2,$ then $x \twoheadrightarrow y.$ This follows from 
	 Lemmas \ref{F*} and  \ref{normalizer}.
\end{proof}

\begin{lemma}\label{cru} Let $X$ be a finite group. Assume that $F^*(X)=F(X)$ and that $X/F(X) \cong S^t,$ where $S$ is a finite non-abelian simple group. Then $\Delta(X)$ is strongly connected.
\end{lemma}
\begin{proof}The statement is trivially true if $Z_\infty(G) \neq 1.$ So we may assume  $Z_\infty(G)=1.$
	
	It follows from the Brauer-Suzuki theorem that there is no simple group with generalized quaternion Sylow 2-subgroups. In particular, by Lemma \ref{frob}, $X/F(X)$ cannot be isomorphic to a Frobenius complement and consequently $X$ is not a Frobenius group.
By Corollary \ref{invo}, all the involutions of $X$ belong to the same strong component,  say $\Omega,$ of $X$. Moreover if $f$ is a non-trivial element of $F(X)$ and $z$ is  an involution of $X$, then $f \twoheadrightarrow z$ by Lemma \ref{normalizer} and $z\mapsto f$ by Lemma \ref{trivial}, so $F(X)\setminus \{1\} \subseteq \Omega.$

 We claim that for any odd 
 prime $p$ dividing $|S|,$
$\Omega$ contains all the elements of $X$ with order $p.$ If $t > 1,$ then a Sylow $p$-subgroup of $S^t$ cannot act fixed-point-freely on $F(X)$, so any Sylow $p$-subgroup of $X$ contains an element in $\Omega,$ and, consequently, all the elements of this Sylow subgroup belong to $\Omega$ (using the fact that the center of a Sylow subgroup is non-trivial). Assume $t=1$ and that, by contradiction, there exists no element of order $p$ in $\Omega$ and let $p$ be the smallest (odd) prime with this property. Let $P$ be a Sylow $p$-subgroup of $S$. Since $P$ acts fixed-point-freely on $F(X)$,
$P = \langle u \rangle$ is cyclic. 
By Burnside's Theorem \cite[Theorem 10.1.8]{rob}, as $S$ is not $p$-nilpotent, we have that $P\not\leq  Z(N_S(P))$, 
 so there exists an element $v \in N_S(P)\setminus C_S(P)$. In particular,  $vC_S(P)$ is a nontrivial element of $N_S(P)/C_S(P)\leq {\rm Aut}(P)$, so its order divides $\varphi(p^n)=p^{n-1}(p-1)$ and is coprime to $p$ (since $P\leq C_G(P)$). But then there is a prime $q$ dividing $(|v|,p-1)$ and so the minimality of $p$ implies that  $v$ is contained in $\Omega$. Since $v$ normalizes $\langle u \rangle$, $v \to u.$ On the other hand, $u \to f$ for any $f\in F(X),$  which means that $u \in \Omega$ and we have reached a contradiction.
\end{proof}

\begin{lemma}
Assume $Z_\infty(G)=1,$ $F(G)=F^*(G)$ and there exist $x, y \in G$  with $d(x,y)>4.$ 
Let $J/F(G)=F(G/F).$ Then 
 $|J/F(G)|$ is prime to $|F(G)|$.
\end{lemma}
\begin{proof}
Let $s$ be a prime divisor of $|F(G)|$ and let $S$ be a Sylow $s$-subgroup of $J$. Take a  power of $y$ of prime order, say $y_r$. 
As $F(G)S$ is normal in $G$, either $y_r$ acts fixed-point-freely on $F(G)S$, and so $F(G)S$ is nilpotent hence contained in $F(G)$, or 
$C_{F(G)S}(y_r) \neq 1$. Let $ g \in C_{F(G)S}(y_r)$ be a non-trivial element of prime power order.  Thus either $g \in F(G)$, and 
  \[ x \mapsto g \mapsto y_r \mapsto y\]
  or $g $ is an $s$-element and it is contained in a Sylow $s$-subgroup containing $S \cap F(G) \neq 1$.  Thus $g$  is contained in a nilpotent subgroup having nontrivial intersection with $F(G)$, hence there exists a nontrivial element $f \in F(G)$ such that 
$[f,_n g]=1$. It follows that 
 \[ x \mapsto f \mapsto g \mapsto y_r  \mapsto y\]
against the assumptions. 
\end{proof}

\begin{lemma}\label{centr}
	Assume $Z_\infty(G)=1,$ $F(G)=F^*(G)$ and there exist $x, y \in G$  with $d(x,y)>4.$ 
	Let $J/F(G)=F(G/F(G)),$ $J^*/F(G)=F^*(G/F(G))$ and $Z/F(G)=Z(J/F(G))$. Let $y_r$ be a power of $y$ of prime order $r.$
If $G$ is not Frobenius and $J=J^*$, then $y_r \in Z$. 
\end{lemma}
\begin{proof}
Suppose $y_r \notin Z$. As $J/F(G)=F^*(G/F(G))$, then $y_r \notin C_{G/F(G)}(J/F(G))$, and so $y_r$ does not centralize a Sylow $t$-subgroup of  $J/F(G)$. 

Assume by contradiction $t=r$ and let $R$ be a Sylow $r$-subgroup of $J\langle y_r \rangle$ containing $y_r$. As $y_r$ does not centralize an $r$-subgroup of $J/F(G)$, $y_r$ is not central in $R$ and the center of $R$ provides another element of order $r$ that centralizes $y_r$. Since, by Lemma \ref{normalizer}, the Sylow subgroups of $ C_G(y_r)$ are cyclic, we get a contradiction. Thus we have
 $t \neq r$. 

We now claim that there exists  a Sylow $t$-subgroup $T$ of $J$ which is  $y_r$-invariant. Notice that 
$y_r$ acts on $J= \prod_p T_pF(G)$ and  centralizes the $r$-Sylow subgroups. If $r$ does not divide $|F(G)|$ then the action is coprime, so there exists an invariant $t$-Sylow. If $r$ divides $|F(G)|$, then $y_r$ belongs to a $r$-Sylow subgroup of $F(G)\langle y_r \rangle$, and there exists a nontrivial element $f$ of $F(G)$ such that $[f, _n y_r]=1$. Thus 
$$x \mapsto f \mapsto y_r  \mapsto y,$$
 a contradiction. 
 
 Let  $T$ be a Sylow $t$-subgroup  of $J$ which is  $y_r$-invariant. Note that $y_r$ acts faithfully on $T$.
 
 Let $V$ be a minimal normal subgroup of $G$. Thus $V \le F(G)$.
 We claim that $C_G(V) \le F(G)$.  Otherwise $y_r$ could not act fixed-point-freely on $C_G(V)$, and consequently 
 there would exist  $g \in C_G(V) \cap C_G(y_r)$ and thus, for any $1 \neq v \in V$, 
   \[ x \mapsto v \mapsto  g  \mapsto y_r \mapsto  y.\]

Moreover, $V$ is a faithful $Q\langle y_r \rangle$-module for every $\langle y_r \rangle$-invariant subgroup $Q$ of $T$. Indeed  $C_G(V) \le F(G)$ and $|F(G)|$ is prime to the order of $T\langle y_r \rangle$. 

Since $C_V(\langle y_r \rangle)$ is trivial, by \cite[Lemma 2.5]{parker} we deduce that $T$ is a $2$-group and that every abelian characteristic  subgroup of $T$ is centralized by $y_r$. 
 Since $C_G(y_r)$ has odd order, we get a contradiction. 
\end{proof}

\begin{proposition}\label{soluble}
	Assume $Z_\infty(G)=1,$ $F(G)=F^*(G)$ and there exist $x, y \in G$  with $d(x,y)>4.$ 
	Let $J/F(G)=F(G/F(G))$ and $J^*/F(G)=F^*(G/F(G))$. 
	If $J=J^*$, then $G$ is a Frobenius group. 
\end{proposition}

\begin{proof}We use the notation introduced in the previous lemma.
Since  $y_r \in Z$,  in particular  $y_r$ belongs to the center of a Sylow $r$-subgroup  $ R$ of $J$. Thus $R \le C_G(y_r)$ is cyclic, and $\langle y_r \rangle$ is a characteristic subgroup of $R$. Now, $RF(G)$ is normal in $G$ and  $\langle y_r \rangle F(G)$ is a characteristic subgroup of $RF(G)$. Thus  $\langle y_r \rangle F(G)$ is normal in $G$. By the Frattini argument, as $y_r$ has order prime to $F(G)$,
\[ G=N_G( \langle y_r \rangle) \left( \langle y_r \rangle F(G) \right) =N_G( \langle y_r \rangle) F(G) \]
hence $G$ is Frobenius by Lemma \ref{normalizer}.
\end{proof}

The  first part of  Theorem \ref{diam4} follows directly from Lemma \ref{reduce} and Proposition \ref{soluble}. 
 
\begin{corollary}\label{diametro}
If $G$ is a soluble  and  $G/Z_\infty(G)$ is not a  Frobenius group, then $\Gamma(G)$ is strongly connected and has diameter at most $4$.
\end{corollary}

Now we are ready to prove our main result. 
\begin{proof}[Proof of Theorem \ref{mainth}] 
Let $G$ be a finite group such that the factor $G/Z_\infty(G)$ is not an almost simple group. By Lemma \ref{reduce}, we can assume that $Z_\infty(G)=1$. 
 In Lemma \ref{frob-graph} we have seen that if $G$ is a Frobenius group then $\Gamma(G)$ is not strongly connected. So we are left to prove that if 
 $G$ is neither almost simple nor Frobenius, then $\Gamma(G)$ is strongly connected. By Lemma \ref{F*} and Proposition \ref{soluble} we are reduced to the case where  $F^*(G)=F(G)$ and $J\neq J^*$.
 If $J=F(G),$ then 
  $\Delta(J^*)$ is connected by Lemma \ref{cru}. In this case $\Gamma(G)$ is strongly connected by Lemma \ref{norma}.

So we are left with the case   $J > F(G).$ Since $J$ is not nilpotent, by Lemma \ref{norma} it is sufficient to prove that  $J\setminus\{1\}$ 
 is contained in a strong component of $\Gamma(G)$. 
 
 If $J$ is not a Frobenius group, then $\Gamma(J)$ is strongly connected, 
 by Proposition \ref{soluble}, and we are done.
 So $J$ is a Frobenius group. 
Let $P/J$ be a Sylow 2-subgroup of $J^*/J.$
 Note that $P/F$ is a central product of a Sylow $2$-subgroup of $E(G/F(G))$ and $J/F(G)$.

We distinguish two cases:

\noindent a) $2$ divides $|F(G)|.$ In this case 
$P$ is neither Frobenius nor almost simple,
so,  by Proposition \ref{soluble},
 $\Gamma(P)$ is strongly connected. As a consequence $J\setminus\{1\}$ 
  is contained in a strong component of $\Gamma(G)$.

\noindent b) $2$ does not divide $|F(G)|$.  In this case 
$F(G)$ has a complement $L$ in $P$ and any element 
of $L$ has centralizer of even order, since $P/F(G)$ is a central product of a Sylow $2$-subgroup of $E(G/F(G))$ and $J/F(G)$.
 As $G$ is neither  Frobenius  nor  almost simple, from Corollary \ref{invo} it follows that $J\setminus\{1\}$ is contained in a strong component of $\Gamma(G)$, as claimed.
 \end{proof}

Indeed a stronger statement can be claimed.

\begin{theorem}\label{gamma2}Suppose $Z_\infty(G)=1.$ If $\Gamma_2(G)$ is not strongly connected, then $G$ is either almost simple or Frobenius.
\end{theorem}
\begin{proof}The proof of Theorem \ref{mainth} relies on the following observations:
	\begin{itemize}
		\item If $[x,y]=1,$ then $x\mapsto y.$
		\item If $x \in N_G(\langle y \rangle),$ then $[x,_2y]=1,$ hence $x\mapsto y.$
		\item If $y\in F(G),$ then $x\mapsto y.$
	\end{itemize}
Clearly if either  $[x,y]=1$ or $x \in N_G(\langle y \rangle),$ then $x\mapsto_2 y.$ If $z \in Z(F(G))$, then $[x,_2z]=1,$ so $x, z, y$ is a path in $\Gamma_2(G)$ for any $y \in F(G).$
\end{proof}

\begin{remark}\label{osser}\rm{The previous theorem is no more true if the assumption $Z_\infty(G)=1$ is removed.
Consider for example $G=\gl(2,3).$ We have $Z_\infty(G)=Z(G) \cong C_2$ and $G/Z(G) \cong S_4.$ By Lemma \ref{diametro}, $\Gamma(G)$ is strongly connected. However, if $x$ has order 3 or 6 and $x\mapsto_2 y$, then $y$ has order $3$ or $6$. Indeed there are $9$ (non-central) 
 elements $g \in G$ with $x\mapsto g$: the three elements of order 3 or 6  different from $x$ in the unique cyclic subgroup of $G$ of order $6$ containing $x$, and $6$ elements of order $4$. If $g$ is one of these $6$ elements of order $4$, then $[x,_3g]= 1$ but
$[x,_2g]\neq 1.$ So $\Gamma_3(G)$ is connected, but $\Gamma_2(G)$ is not}.
\end{remark}

\section{An  example}\label{esempio}\label{sec:4}

The aim of this section is to complete the proof of Theorem \ref{diam4}: we will present a family of soluble groups $G$ such that the diameter of  $\Gamma(G)$ 
 is $4$.  

 We select $q$ a power of an odd prime such that the prime $r,$ with $r\geq 3$, divides $q-1$ exactly. Let $\mathbb F=\GF(q^r)$ and let $\beta$ the Frobenius automorphism of $\mathbb F$ of order $r$. Note that $r^2$ divides exactly $q^r-1$ (see \cite[Lemma 8.1 (e)]{hb2}).  
 Let $t$ be a prime that divides $(q^r-1)/(q-1)$ but not $q-1$. As an example we may take $q=7, r=3, t=19.$
 Let
 $$z:=\begin{pmatrix}e&0\\0&e\end{pmatrix}\in \gl(2,\mathbb F),$$
 with $|e|=r^2.$
 and let
  $$c:=\begin{pmatrix}1&0\\0&f\end{pmatrix}\in \gl(2,\mathbb F),$$
 with $|f|=t.$
 Recall that $\beta$ induces an automorphism of the affine group $\AGL(2,q^r)\cong F \rtimes \gl(2,q^r)$, with $F\cong \mathbb F^2.$ We also denote this induced automorphism by $\beta.$ Note that
 $\beta$ normalizes $F, \langle z \rangle$ and $\langle c\rangle.$ Consider $x=z\beta$ in the semidirect product $\AGL(2,q^r)\rtimes \langle\beta\rangle$ and let 
$$G=F \rtimes D, \text { with } D=\langle x, c\rangle.$$

	\begin{lemma}
	The element $x^r$ has order $r$, commutes with $c$ and acts fixed-point-freely on $F$. Moreover $c^x=c^q\neq c.$
	\end{lemma}
\begin{proof}
We have
$$x^r=(z\beta)^r=z^{1+q+\dots+q^{r-1}}=\begin{pmatrix} e^{1+q+\dots+q^{r-1}}&0\\0&e^{1+q+\dots+q^{r-1}}\end{pmatrix}.$$
Since $|e|=r^2$ and $r^2$ does not divides $(q^r-1)/(q-1),$ it follows that
$x^r$ has order $r$ and acts fixed-point-freely on $F$. The other information in the statement can be easily verified.
\end{proof}

By the previous lemma $|G|=q^{2r}\cdot t\cdot r^2$ (so in particular $|G|= 7^6\cdot 19\cdot 3^2$ if $q=7, r=3, t=19).$

	\begin{lemma}The following statements hold.
		\begin{enumerate}
			\item $C_G(x)=\langle x \rangle.$
			\item $C_G(x^r)=D.$
			\item Let $C=\langle c \rangle.$ Then $[F,C]$ is a non-trivial subgroup of $F$ normalized by $D$ and $[F,C]D$ is a Frobenius group.
		\end{enumerate}
		\end{lemma}

	\begin{proof} 
		\begin{enumerate}
			\item Let $g=fd\in C_G(x)$ for some $f\in F$ and $d\in D$. So, $1=[fd,x]=[f,x]^d[d,x]$, and therefore $[f,x]=1$ and $[d,x]=1$. Since $x$ acts fixed-point-freely on $F$,  we deduce $f=1$. Moreover $c^x=c^q\neq c$, hence $C_G(x)=C_D(x)=\langle x\rangle$.
			\item Arguing as before, since $x^r$ acts fixed-point-freely on $F$, we deduce $C_G(x^r)=C_D(x^r)=D.$
			\item Since $C\unlhd D,$ $[F,C]$ is normalized by $D$. Notice that
			$[F,C]$ is precisely the eigenspace of $c$ corresponding to the eigenvalue $f.$ In particular $[F,C]$ is a 1-dimensional $\mathbb F$-subspace of $F.$
				Let $1\neq d \in D.$ If $r$ divides $|d|$, then $x^r \in \langle d \rangle,$ so
			$C_{[F,C]}(d)\leq C_{[F,C]}(x^r) \leq  C_{F}(x^r)=1.$ If $r$ does not divide $|d|,$ then $\langle d \rangle=\langle c \rangle,$ so 	$C_{[F,C]}(d)=C_{[F,C]}(c)=1.$\qedhere
		\end{enumerate}
	\end{proof}

	We denote by $ \B_i(\{x\})$ the set of vertices of $\Gamma (G)$ having distance at most $i$ from $x$. 
	By Corollary \ref{diametro} 
	 the diameter of  $\Gamma(G)$  is at most $4$. We will show that   diam$(\Gamma(G))=4$, by proving, in Lemma \ref{non3},  that $\B_3(\{x\}) \neq G \setminus{\{1\}}$.  
	
	\begin{lemma} $ \B_1(\{x^i\})=\langle x \rangle \setminus \{1\},$ for every $i$ prime to $r$. 
	\end{lemma}
	\begin{proof}Let $g \in  \B_1(\{x\}).$ Then $[g,_nx]=1$ for some positive integer 
		$n.$ Choose $n$ as small as possible.
		If $n>1,$ then let
		$s=[g,_{n-1}x].$ It follows $s \in G^\prime \cap C_G(x) = FC \cap \langle x \rangle=1,$ a contradiction with the minimality of $n.$ 	So $n=1,$ and $g\in C_G(x) =\langle x \rangle.$ The same arguments apply whenever $x^i$ is a generator of $\langle x \rangle$. 
	\end{proof}
	
	\begin{lemma} $ \B_2(\{x\})=D\setminus \{1\}.$
	\end{lemma}
	\begin{proof}By the previous lemma, $ \B_2(\{x\}) \subseteq\langle x \rangle \cup  \B_1(\{x^r\}).$
	 Assume, by contradiction,  that there exists an element  $g \in  \B_1(\{x^r\})$ that does not belong to $D=C_G(x^r)$. 
	  Then $[g,_nx^r]=1$ for some $n \ge 2$ and $[g,_{n-1}x^r] \in C_G(x^r)=D$. 
		 Choose the smallest $m$ with the property that $[g,_mx^r]\in D$ and note that  $m \ge 1$. 
		 Let us set $w=[g,_{m-1}x^r].$ 
		  So, $w \notin D$ while $[w,x^r] \in D$. 
		 Write $w=fd \in G=FD$, for some $f \in F$ and $d \in D$. 
		 Then $[w,x^r]=[f,x^r]^d[d,x^r] \in D$ implies $[f,x^r] \in D \cap F=1$. So $f \in C_G(x^r) \cap F=1$ and hence $w \in D$, a contradiction.  
	\end{proof}

\begin{lemma}\label{non3} $[F,C] \setminus \{1\} \not\subseteq  \B_3(\{x\}).$
\end{lemma}

\begin{proof}By the previous lemmas, if $y \in  \B_3(\{x\}) \cap [F,C]$, then
	$y \mapsto d,$ for some non-trivial element $d\in D.$ But, by Lemma \ref{frob-graph}, this is impossible since $[F,C]D$ is a Frobenius group.
\end{proof}

\section{Almost simple groups}	\label{sec:5}

\subsection{Alternating group}\label{alterni}
In this subsection we investigate the strong connectivity of the alternating groups. One of the tool that we can use is the knowledge of the connected components of the prime graph of a finite simple group \cite{wi}.  Recall that the prime graph of a finite group $G$ is the graph whose vertices are the primes dividing the order of $G$ and where two vertices $p$ and $q$ are joined by an edge if and only if $G$ contains an element of order $p\cdot q.$

\begin{lemma}\label{3ciclo}
If $n\geq 6,$ then
$$(1,3,5)(1,\dots,n)(1,3,5)^{-1}(1,\dots,n)^{-1}=(1,3,5)(2,n,4).$$
\end{lemma}

\begin{theorem}\label{alte} If $n\geq 5,$ then $\Gamma(A_n)$ is strongly connected if and only if $n\neq 5.$ Moreover the following holds:
\begin{itemize}
\item 	if $n\geq 7$, then $\Gamma_2(A_n)$ is strongly connected;
	\item $\Gamma_3(A_6)$ is strongly connected, but $\Gamma_2(A_6)$ is not.
\end{itemize}
\end{theorem}

\begin{proof} Let $G=A_n.$ All the elements of order $2$ belong to the same strong component of $\Gamma_1(G)$ (see  \cite{ij} or the first lines of the proof of \cite[Theorem 7.1]{mp}). Let  $\Omega$ be the strong component of $\Gamma_2(G)$ containing all the elements of order $2$ and let $\pi_1(G)$ be the connected component of the prime graph of $G$ containing 2. Clearly $\Omega$ contains also all  the elements whose order is divisible by a prime $q\in \pi_1(G),$ as every element commutes with its own powers.
	
	 Suppose $n\geq 8.$ By \cite[Table I]{wi} 
either the prime graph of $G$ is connected or there exists a prime $p$ such that $n\in \{p, p+1, p+2\}$ and $\pi(G)=\pi_1(G) \cup \{p\}.$ In the first case we can immediately conclude that $\Gamma_2(G)$ is connected. In the second case, we need to prove that if $|x|=p,$ then $x\in \Omega.$
By Lemma  \ref{3ciclo}, there exists a 3-cycle $y$ with $[x,y,y]=1$,  so	$x \twoheadrightarrow_2 y,$  with $y\in \Omega$. 
Moreover  there exists an element $z$ of order $(p-1)/2$ which normalizes $\langle x \rangle$. In particular $z\in \Omega$ and $z \mapsto_2 x.$

Assume  $n=7$. Again all the elements of even order of $G$ belong to the same strong component $\Omega$ of $\Gamma_2(G).$ A $3$-cycle commutes  with a double-transposition, so $\Omega$ contains all the $3$-cycles, as well as all the products of two disjoint $3$-cycles. 
 Moreover $(2,5)(3,4) \mapsto_2 (1,2,3,4,5)$ and $(1,2,3,4,5)\mapsto_2 (3,5)(6,7)$ so $\Omega$ contains all the $5$-cycles. Finally $(1,2,3,4,5,6,7)\mapsto_2 (1,4)(3,7)$ and $(2,3,5)(4,7,6)\mapsto_2 (1,2,3,4,5,6,7).$ So,  $\Gamma_2(A_7)$ is strongly connected

Now, let $n=6.$ If $x =(1,2,3,4,5)$ and $x\mapsto_2 y$, then $y \in \langle x \rangle$ and therefore $\Gamma_2(G)$ is not strongly connected. But $x\mapsto_3 (2,4)(5,6)$ and 
$(2,5)(3,4)\mapsto_3 x.$ So there is a strong component $\Omega$ of $\Gamma_3(G)$ containing all the elements of $G$ of even order and all the $5$-cycles. The other elements of $G$ are conjugated in $\Aut(G)$ to $(1,2,3)$ so to conclude it suffices to notice that $(2,3)(5,6)\mapsto_2 (1,2,3)$ and $(1,2,3)\mapsto_2 (1,2)(3,4).$

Finally assume $n=5.$ If $x =(1,2,3,4,5)$ and $x\mapsto y$, then $y \in \langle x \rangle$ and therefore $\Gamma(G)$ is not strongly connected.
\end{proof}

\subsection{Symmetric groups}

The aim of this subsection is to prove that if $n\geq 5,$ then $\Gamma_2(S_n)$ is strongly connected. When $n\geq 7,$ this could be immediately deduced combining Lemma \ref{norma} with Theorem \ref{alte}. We prefer to give an easy and self-contained proof.
\begin{lemma}\label{invol}
	If $x=(a,b)$ and $y=(c,d)$ are two transpositions in $S_n,$ then $x \twoheadrightarrow_1   y.$
\end{lemma}
\begin{proof}It suffices to notice that  either $[x,y]=1$ or there exists a transposition $z=(e,f)$ with $[x,z]=[y,z]=1.$
\end{proof}

\begin{lemma}\label{real}
	Let $1 \neq x\in S_n.$ Write $x=\sigma_1 \cdots \sigma_r$ as product of disjoint cycles (including possibly cycles of length 1), in such a way that $|\sigma_1|\leq \dots \leq |\sigma_r|.$ If $|\sigma_r|\leq n-2,$ then there exists a transposition $u$ such that $x \twoheadrightarrow_1 u$ and $u \twoheadrightarrow_1 x.$
\end{lemma}
\begin{proof} Let $y=	\sigma_1 \cdots \sigma_{r-1}$. If $y \neq 1$, then set $u=(a,b)$ with $\{a,b\}\subseteq {\rm{supp}}(\sigma_r)$, so that 
	$[x,y]=1$ and $[y,u]=1$. Otherwise, take $u=(a,b)$ with $\{a,b\} \cap {\rm{supp}}(\sigma_r)=\emptyset.$
\end{proof}

\begin{lemma}\label{inverte}
	If $1 \neq x\in S_n,$ then there exists an involution $u$ such that $u \twoheadrightarrow_2 x.$
\end{lemma}
\begin{proof}
	Note that	$x$ is conjugate to $x^{-1}$ and consequently $N_G(\langle x\rangle)$ contains an involution.
\end{proof}

\begin{lemma}\label{conto}If $m\geq 4,$ then
	$$	(1,3)\cdot (1,2,3,\dots, m)\cdot
	(1,3)\cdot (1,2,3,\dots, m)^{-1}=(1,3)(2,m)$$ and consequently
	$[ (1,2,3,\dots, m),(1,3),(1,3)]=1.$
\end{lemma}

\begin{theorem}If $n\geq 5,$ then $\Gamma_2(S_n)$ is strongly connected.
\end{theorem}

\begin{proof}
	By Lemmas \ref{invol} and \ref{real}, all the elements of $G$ that are neither a $(n-1)$-cycle nor a $n$-cycle belong to the same 
	strong component of $\Gamma_1(G)$, and consequently, to the same strong component, say $\Omega,$ of $\Gamma_2(G).$ 
	Suppose
	that $x$ is an $r$-cycle, with $r \in \{n-1,n\}.$ 	By Lemma \ref{conto},
	$x \twoheadrightarrow_2 u$ for a suitable involution $u.$ Combining this with Lemma \ref{inverte}, we conclude $x\in \Omega.$
\end{proof}

\subsection {Simple groups whose Engel graph is not strongly connected}\label{sconnessi}
It follows from the previous subsections that there are infinitely many simple and almost simple groups with a strongly connected Engel graph. Now our aim is to prove that there exist also infinitely many simple groups whose Engel graph is not strongly connected. We first need a preliminary lemma.

\begin{lemma}\label{NC}
	Let $y$ be contained in a subgroup $ K \le G$ with the property that 
	\begin{enumerate}
		\item $N_G(K)=K$;
		\item $y\in K^g$ if and only if $K^g=K$.
	\end{enumerate}
	Then $x\notin K$ implies $x\not\mapsto y$.
\end{lemma} 
\begin{proof} 
	Note that conditions (1) and (2) imply that $y^{-g} \in K$ if and only if $g \in K$.
	Assume $x \mapsto y$, so that $[x,_n y]=1$ for some $n \geq 1$, and 
	let $m$ be the minimal integer such that $[x,_m y]\in K$. 
	Assume $m >0$. 
	 We set $w=[x,_{m-1}y]$. 
	 Then  from $[x,_m y]=[ w ,y]=y^{-w} y \in K$ and $y \in K$, we deduce that 
	$y^{-w} \in K$. So,  $w \in K$, against the minimality of $m$. It follows that $m=0$, that is $x \in K$.  
\end{proof}

\begin{theorem}Let $G=\PSL(2,q)$ with $q=2^f.$ Then $\Gamma(G)$ is not strongly connected. 
\end{theorem}
\begin{proof} Note that the order of an element of $G$ can be:  
	\begin{itemize}
		\item 2;
		\item a divisor of $q-1;$
		\item a divisor of $q+1.$
	\end{itemize}
	Now let $x\neq 1$ be an element whose order divides $q+1.$ We are going to prove that if $x \mapsto y$, then $|y|$ divides $q+1.$	Let $y$ be a nontrivial element of $G$. 
	
	\noindent (a)  
	Assume $|y|=2$ and let  
	 $P$ be a Sylow 2-subgroup of $G$, with $y \in P.$ Then $P \cong C_2^f$ and $K=N_G(P) \cong P \rtimes C_{q-1}$ is a maximal subgroup of $G.$ We may apply Lemma \ref{NC} to conclude that $x\not\mapsto y.$
	
	\noindent (b)  Assume that  $ |y|=t$ divides $q-1$ and set 
	 $K=N_G(\langle y \rangle).$ Then $K$ is a maximal subgroup of $G$ and $K\cong D_{2(q-1)}.$ Again we apply Lemma \ref{NC} to conclude that $x\not\mapsto y.$
	\end{proof}

\begin{theorem}
	Let $G=\sz(q)$ with $q=2^{2t+1}$ and $t\geq 1$. Then $\Gamma(G)$ is not strongly connected.
\end{theorem}

\begin{proof}
	If $g \in G$ then the order of $G$ can be:
	\begin{itemize}
		\item $2$ or $4$;
		\item a divisor of $q-1;$
		\item a divisor of $q+\sqrt{2q}+1;$
		\item a divisor of $q-\sqrt{2q}+1.$
	\end{itemize}
	Now let $x\neq 1$ be an element whose order divides $q+\sqrt{2q}+1.$ We are going to prove that if $x \mapsto y$, for a nontrivial element $y \in G$,  then $|y|$ divides $q+\sqrt{2q}+1.$
	
	\noindent (a)  If $|y|$ is $2$ or $4,$ then  $y$ is contained in a unique Borel maximal subgroup, since the intersection of two distinct such subgroups does not contain $2$-elements. Therefore by Lemma \ref{NC} we conclude $x\not\mapsto y$.
	
	\noindent (b) If $|y|$ divides $q-1$,  let $K=N_G(C_G(y))$. This is a maximal subgroup of dihedral type, so it is self-normalizing. We have $\langle y\rangle\nn K$. If $y\in K^g$ with $K^g\neq K$, then $\gen{y}\nn K^g$ since it is the unique subgroup of $K^g$ of order $|y|$. Therefore $\gen{y}\nn\gen{K,K^g}=G$, which is impossible, so $K$ respects the requirements of Lemma \ref{NC} and $x\not\mapsto y$.
	
	\noindent (c) If $|y|$ divides $q-\sqrt{2q}+1$, 
	let $K=N_G(C_G(y))$. $K$ is a maximal subgroup of type $C_{q-\sqrt{2q}+1}\rtimes C_4$ and so it is self-normalizing. Since $\gen{y}$ is the unique subgroup of $K$ of order $|y|$, we can argue as in (b) and conclude $x\not\mapsto y$.
	\end{proof}
\section{Some computations}\label{sec:6}

Combining the available information about the connected components of the prime graph of $G$ (see \cite{wi}) with some  direct computations with GAP \cite{GAP4}, we checked that $\Gamma(G)$ is strongly connected if $G$ is a sporadic simple group, although its commuting graph $\Gamma_1(G)$ is not (notice that by Lemma \ref{norma} this implies that $\Gamma(\Aut(G))$ is also strongly connected). 
Indeed, let $\pi_1(G)$ be the connected components  of the prime graph of $G$ which contains 2. By \cite[Lemma 6.1]{mp}, there is a unique connected component of the commuting graph $\Gamma_1(G)$ containing all the elements of even order, and consequently there is a strong component of $\Gamma(G)$ containing all the elements whose order is divisible by a prime $p\in \pi_1(G).$ In particular, $\Gamma(G)$ is strongly connected if for any prime divisor $q$ of $|G|$ with
$q\not\in \pi_1(G),$ there exists a subgroup $H$ of $G$ whose order is divisible by $q$ and some prime in $\pi_1(G)$ and such that $\Gamma(H)$ is strongly connected. 

For example, the prime graph of $\rm{J}_2$ has two connected components, $\pi_1(\rm{J_2})=\{2,3,5\}$ and $\{7\},$ so it suffices to notice that $\rm{J}_2$ contains a subgroup $H\cong \PSL(2,7)$ and to check with GAP that $\Gamma(\PSL(2,7))$ is strongly connected.

 Similarly if $G=\rm{M},$ then 41, 59 and 71 are the only prime divisors of $G$ not contained in $\pi_1(G)$; since $\rm{M}$ contains subgroups isomorphic to $\PSL(2,41), \PSL(2,59)$ and $\PSL(2,71)$ (see \cite{atlas} and \cite{41}) we reduce to check that the Engel graphs of these subgroups are strongly connected.
 
  Another case worth mentioning is $G = \rm{J_1}$.
In that case the prime divisors of $|G|$ not contained in $\pi_1(G)$ are 7, 11, 19.
However $G$ contains a maximal subgroup $M_1\cong 2^3:7:3$ and a maximal subgroup $M_2\cong \PSL(2,11)$ and $\Gamma(M_1), \Gamma(M_2)$ are strongly connected. There is also a maximal subgroup isomorphic to 19:6 and consequently, by Lemma \ref{trivial}, $\Gamma(G)$ contains and edge from an element of order 6 to an element of order 19. Hence to prove that $\Gamma(G)$ is strongly connected it suffices to prove that there is an edge $x \mapsto y$, with $|x|=19$ and $|y|\neq 19.$ When we checked whether this is true, we found an unexpected phenomenon: there exists $g$ of order $19$ such that $g\mapsto y$ only if $y \in \langle g \rangle,$ but $g^2$ has a different behavior, since there are 38 involutions $z$ with $g\mapsto z.$ 

The four  sporadic groups $\{\rm {J_4, Ly, F_{24}^\prime, B}\}$ are more difficult to handle. In these cases the methods described above do not allow to deduce that there is an edge $x\mapsto y$ with $|x|=p$ and $|y|$ divisible by a prime in $\pi_1(G),$ when $(G,p) \in \{(\rm{J_4},29),  (J_4,43), (Ly,37),$ $(Ly,67), (F_{24}^\prime,29), (B,47)\},$ since $\Gamma(H)$ turns out to be disconnected for every subgroup $H$ of $G$ whose order is divisible by $p.$ In that case the following remark is of great help. As it can be easily verified, if $g_1,g_2$ are elements of an arbitrary group and $|g_2|=2,$ then $[g_1,_ng_2]=[g_1,g_2]^{(-2)^{n-1}}$. In our case, testing the commutators between a randomly chosen element of order $p$ and a random involution, we produce pair $(x,y)$ with $|x|=p, |y|=2$ and $[x,_3y]=[x,y]^4=1.$

We also used GAP to check that if $p$ is and odd prime with $p\leq 109,$ then the graph $\Gamma(\PSL(2,p))$ is strongly connected if and only if
$p \notin \{13, 29, 37, 53, 61, 101, 109\},$ i.e. if and only if $p\neq 5 \mod 8.$
We aim to investigate whether this is the general behavior in a subsequent paper.
 It is interesting to notice that
$G=\PSL(2,47)$ is the smallest simple group we know, with the property that $\Gamma_4(G)$ is strongly connected, but $\Gamma_3(G)$ is not. If $G=\PSL(2,127),$ 
then $\Gamma_7(G)$ is strongly connected, but $\Gamma_6(G)$ is not.

To perform these computations we wrote a function \lq\lq$\rm{eng}$" defined on the elements of $G$. This function, given two nontrivial elements 
 $x, y \in G$,  recursively 
 evaluates the commutators $[x,y], [x,_2 y] \dots$ and stops when either $[x,_n y]=1$ or $ [x,_{n_1} y]  = [x,_{n_2} y] \neq 1$ with $n_1\neq n_2.$ 
 The first component $\rm{eng}(x,y)[1]$ of the output of $\rm{eng}(x,y)$ is  either the trivial element,  if $[x,_ny]=1$ for some positive integer, or  a nontrivial  element $z$ with the property that $z= [x,_{n_1} y]=[x,_{n_2} y]$ with $n_1\neq n_2.$ The second component $\rm{eng}(x,y)[2]$  keeps track of the commutators that have been evaluated during the procedure. From Lemma \ref{trivial} (1) it follows that if $ 1 \neq [x,_{n_1} y]  = [x,_{n_2} y]$ with $n_1\neq n_2$,  then $x$ is not adjacent to $y$. 
 Therefore, $x\mapsto y$ if and only if $\rm{eng}(x,y)[1]=Identity(Group(x,y))$. In particular $x\mapsto_n y$ if and only if 
$\rm{eng}[1]=Identity(Group(x,y))$ and $\rm{Size(eng}(x,y)[2]) \leq n.$

\begin{lstlisting}
eng:=function(x,y)
local z,s;
s:=[Identity(Group(x,y))];z:=Comm(x,y);
while (z in s)=false do
Add(s,z);
z:=Comm(z,y);
od;
return[z, s];
end;
\end{lstlisting}

\end{document}